\documentclass[12pt, reqno, a4paper]{amsart}
\usepackage{times}

\usepackage[utf8]{inputenc}
\usepackage[USenglish]{babel}
\usepackage{amsmath,amsthm,amssymb,amsfonts}
\usepackage{mathtools}
\usepackage{mathrsfs}
\usepackage{bbm}

\usepackage{indentfirst}
\usepackage{enumitem}
\usepackage{xcolor}
\usepackage{graphicx}

\usepackage{float}

\usepackage[breaklinks=true, bookmarksopenlevel=1, bookmarksdepth=2]{hyperref}
\hypersetup{
colorlinks,
   linkcolor={cyan!80!black},
   citecolor={cyan!80!black},
 urlcolor={cyan!80!black}
}

\allowdisplaybreaks

\newtheorem{thm}{Theorem}[section]

\newtheorem{lem}[thm]{Lemma}

\theoremstyle{plain} % just in case the style had changed
\newcommand{\thistheoremname}{}
\newtheorem*{genericthm}{\thistheoremname}

\theoremstyle{definition}

\theoremstyle{remark}

\numberwithin{equation}{section}

\newcommand{\N}{\mathbb{N}}      % N = Naturals
\newcommand{\Z}{\mathbb{Z}}      % Z = Integers
\newcommand{\R}{\mathbb{R}}      % R = Reals
\newcommand{\eps}{\varepsilon}   % epsilon

\newcommand{\ul}[1]{\mathbf{#1}}

\usepackage[margin=1in]{geometry}

\restylefloat{table}

\setlist[itemize]{leftmargin=*}
\setlist[enumerate]{leftmargin=*}

\begin{document}

\title{Infinite Sidon-type sets for zero-sum linear forms}%

\author{Christian T\'afula}
\address{Institute of Mathematics, Statistics\\
and Computer Science\\
University of S\~ao Paulo\\
Rua do Mat\~ao, 1010\\
S\~ao Paulo, SP 05508-090\\
Brazil}
\curraddr{}
\email{tafula@ime.usp.br}
\thanks{}


% ----------------------------------------------------------------
\begin{abstract}
 Let $h \geq 2$, and let $\mathbf{b} = (b_1,\dots,b_h)\in \mathbb{Z}^h$ be a zero-sum vector with nonzero coordinates. For a set $A=\{a_1<a_2<\cdots\}\subseteq\mathbb{N}$, let $r_{A,\mathbf{b}}(n)$ denote the number of $h$-tuples $(x_1,\ldots,x_h)$ of pairwise distinct elements of $A$ satisfying $b_1x_1+\cdots+b_hx_h=n$. We study density restrictions on sets $A$ for which these representation counts remain small, obtaining analogues of the classical density theorem for infinite Sidon sets.

 In the case $\mathbf{b} = (c_1,-c_1,\dots,c_k,-c_k)$, we prove that if $A(x)/(x/\log x)^{1/2k}\to\infty$, then $\frac{1}{x}\sum_{|n|\leq x} r_{A,\mathbf{b}}(n)\to\infty$, whereas if $A(x)\gg x^{1/2k}$, then $\frac{1}{x}\sum_{|n|\leq x} r_{A,\mathbf{b}}(n)\gg\log x$. This recovers Chen's theorem on $B_{2k}$-sequences. For general zero-sum vectors $\mathbf{b}$, we prove analogous bounds under gap conditions: if $a_{n+1}-a_n=o(n^{h-1}\log n)$, then $\frac{1}{x}\sum_{|n|\leq x} r_{A,\mathbf{b}}(n)\to\infty$, whereas if $a_{n+1}-a_n\ll n^{h-1}$, then $\frac{1}{x}\sum_{|n|\leq x} r_{A,\mathbf{b}}(n)\gg\log x$.
\end{abstract}

\maketitle
% ----------------------------------------------------------------

%%%%%%%%%%%%%%%%%%%%%%%%%%%%%%%%%%%%%%%%%%%%%%%%%%%%%%%%%
\section{Introduction}
 A \emph{Sidon set} (or \emph{$B_2$-set}) is a set $A\subseteq \Z$ in which all pairwise sums are distinct; equivalently, if $a+b=c+d$ with $a,b,c,d\in A$, then $\{a,b\}=\{c,d\}$. In the finite setting, the maximal size $F_2(n)$ of a Sidon set contained in $[1,n]$ satisfies $F_2(n)\sim \sqrt{n}$, with constructions due to Singer, Bose, and Chowla and earlier bounds of Erd\H{o}s and Tur\'an \cite{erdtur41}; see also \cite{halberstam83}.

 The infinite setting is subtler. Throughout, if $A\subseteq \N$, we write
 \[ A=\{a_1<a_2<\cdots\} \]
 for its increasing enumeration, and $A(x):=|A\cap[1,x]|$. Erd\H{o}s showed that there exist infinite Sidon sets $A\subseteq \N$ with $A(x)\gg x^{1/2}$ along an unbounded sequence of $x$, while also proving the complementary bound
 \begin{equation}
  \liminf_{x\to\infty}\frac{A(x)}{(x/\log x)^{1/2}}<\infty. \label{erdos-liminf}
 \end{equation}
 Thus the density of an infinite Sidon set must dip below $((\log x)/x)^{1/2}$ infinitely often. Constructions of relatively dense infinite Sidon sets are also known; for instance, Ruzsa \cite{ruz98} constructed Sidon sets with $A(x)\gg x^{\sqrt{2}-1}$, and $\sqrt{2}-1=0.4142\ldots$ remains the best explicit exponent known.

 More generally, a \emph{$B_h[g]$-set} is a set $A\subseteq \Z$ in which every integer has at most $g$ representations as a sum of $h$ (not necessarily distinct) elements of $A$, counted up to permutation. Counting gives the upper bound $F_h^g(n)\ll_{g,h} n^{1/h}$ for the largest $B_h[g]$-subset of $[1,n]$, while probabilistic constructions show that for every $\eps>0$ there exists $g=g_{h,\eps}$ and a $B_h[g]$-set $A\subseteq \N$ with $A(n)\gg_\eps n^{1/h-\eps}$ (see \cite[\S1.7]{taovu06}). For infinite $B_h$-sets (the case $g=1$), the analogue of \eqref{erdos-liminf} is known for even $h$: Jia \cite{jia94} proved it for $B_{2k}$-sequences under the auxiliary growth condition $A(n^2)\leq A(n)^2$, and Chen \cite{che93} removed this condition. The corresponding problem for odd $h$ remains open.

 In this paper we consider representation constraints coming from zero-sum linear forms. Let $h\geq 2$ and let $\ul{b}=(b_1,\ldots,b_h)\in \Z^h$ with $b_i\neq 0$ for all $i$. For $A\subseteq \N$ and $n\in \Z$ we define the \emph{distinct ordered} representation function
 \begin{equation}
  r_{A,\ul{b}}(n):=\#\{(x_1,\ldots,x_h)\in A^h ~|~ x_i\neq x_j,\ \sum_{i=1}^h b_i x_i=n\}. \label{def-rp}
 \end{equation}
 We say that $A$ is a \emph{weak $B_h[g]$-set with respect to $\ul{b}$} if $r_{A,\ul{b}}(n)\leq g$ for all $n\in \Z$. This includes classical Sidon sets (taking $\ul{b}=(1,-1)$ and ignoring diagonals), and also other additive restrictions described by linear forms; see Nathanson \cite{nat22} for related generalizations.

 Our main results concern the case of \emph{zero-sum} vectors, namely $\sum_{i=1}^h b_i=0$. In this regime the linear form is translation-invariant, and this invariance already imposes strong restrictions on the density of sets with few representations. We prove a density theorem in the matched-even case $\ul{b}=(c_1,-c_1,\ldots,c_k,-c_k)$, and then establish lower bounds for arbitrary zero-sum vectors under hypotheses on the growth of the gaps $a_{N+1}-a_N$. In the special case $\ul{b}=(1,-1)$, our results recover the classical Erd\H{o}s phenomenon \eqref{erdos-liminf}.

 \begin{thm}\label{MTme}
  Let $k\geq 1$ and let $\ul{b}=(c_1,-c_1,\ldots,c_k,-c_k)\in \Z^{2k}$ with $c_i\neq 0$ for all $i$. Let $A\subseteq \N$, and let $r_{A,\ul{b}}(n)$ be defined by \eqref{def-rp}.
  \begin{enumerate}[label=\textnormal{(\roman*)}]
   \item \label{me-i} If $\dfrac{A(x)}{(x/\log x)^{1/2k}}\to\infty$, then $\displaystyle \frac{1}{x}\sum_{|n|\leq x} r_{A,\ul{b}}(n)\to\infty$.\smallskip
   \item \label{me-ii} If $A(x)\gg x^{1/2k}$, then $\displaystyle \frac{1}{x}\sum_{|n|\leq x} r_{A,\ul{b}}(n)\gg \log x$.
  \end{enumerate}
 \end{thm}

 Chen \cite{che93} proved the even-order analogue of \eqref{erdos-liminf}: if $A\subseteq \N$ is a $B_{2k}$-sequence, then
 \begin{equation}
  \liminf_{x\to\infty}\frac{A(x)}{(x/\log x)^{1/2k}}<\infty. \label{chen}
 \end{equation}
 A key input in \cite{che93} is that any $B_{2k}$-sequence is a weak $B_{2k}[(k!)^2]$-set with respect to the alternating zero-sum vector
 \[ \ul{b}=\underbrace{(1,-1,1,-1,\ldots,1,-1)}_{2k\text{ entries}}. \]
 Thus $\sum_{|n|\leq x} r_{A,\ul{b}}(n)\ll x$, and Theorem \ref{MTme} \ref{me-i} implies \eqref{chen} by contraposition. The proof of Theorem \ref{MTme} is based on a lower bound on each dyadic block $\{a_{N+1},\ldots,a_{2N}\}$: a non-negative Fourier integral shows that many $2k$-tuples from this block have $\sum b_i u_i$ lying in a short interval around $0$. Summing these bounds over $N$ in a suitable range yields both parts of the theorem.

 Our second result applies to arbitrary zero-sum vectors, where we obtain lower bounds from hypotheses on the growth of the gaps $a_{N+1}-a_N$.

 \begin{thm}\label{MTbdd}
  Let $h\geq 3$ and let $\ul{b}=(b_1,\ldots,b_h)\in \Z^h$ satisfy $\sum_{i=1}^h b_i=0$ and $b_i\neq 0$ for all $i$. Let $A\subseteq \N$, and let $r_{A,\ul{b}}(n)$ be defined by \eqref{def-rp}.
  \begin{enumerate}[label=\textnormal{(\roman*)}]
   \item \label{bdd-i} If $a_{N+1}-a_N=o(N^{h-1}\log N)$, then $\displaystyle \frac{1}{x}\sum_{|n|\leq x} r_{A,\ul{b}}(n)\to\infty$.\smallskip
   \item \label{bdd-ii} If $a_{N+1}-a_N\ll N^{h-1}$, then $\displaystyle \frac{1}{x}\sum_{|n|\leq x} r_{A,\ul{b}}(n)\gg \log x$.
  \end{enumerate}
 \end{thm}

 By contraposition, Theorem \ref{MTbdd} \ref{bdd-i} implies that if $A$ is a weak $B_h[g]$-set with respect to some zero-sum $\ul{b}$, then
 \begin{equation}
  \limsup_{N\to\infty}\frac{a_{N+1}-a_N}{N^{h-1}\log N}>0. \label{gap-obstruction}
 \end{equation}
 The proof of Theorem \ref{MTbdd} is based on a combinatorial lemma. Assuming control of the gaps on a block $[N,2N]$, we construct many representations with $\sum b_i u_i$ in $[-x,x]$, leading to a lower bound in terms of the local maximum of the gaps on that block. Estimating the resulting sum over $N$ then yields both parts of Theorem \ref{MTbdd}.

 The hypotheses in Theorem \ref{MTbdd} are consistent with the density thresholds suggested by Theorem \ref{MTme}: the condition $a_{N+1}-a_N=o(N^{h-1}\log N)$ implies $A(x)/(x/\log x)^{1/h}\to\infty$, while $a_{N+1}-a_N\ll N^{h-1}$ implies $A(x)\gg x^{1/h}$. This suggests the conjecture that a version of Theorem \ref{MTme} should hold for all zero-sum vectors $\ul{b}$.

 Finally, the zero-sum hypothesis is essential. When $\sum_{i=1}^h b_i\neq 0$, the linear form $\sum_{i=1}^h b_i x_i$ is no longer translation-invariant, and the conclusions of Theorems \ref{MTme} and \ref{MTbdd} can fail. For example, when $\ul{b}=(1,1)$ and $A=\{n^2\}_{n\geq 1}$, one has $a_{N+1}-a_N\ll N$ but $\sum_{n\leq x} r_{A,(1,1)}(n)=O(x)$. There also exist additive bases $A\subseteq \N$ of order $2$ (so $A(x)\gg x^{1/2}$) with $\sum_{n\leq x} r_{A,(1,1)}(n)=O(x)$ (see Nathanson \cite{nat12}), and even with bounded second moment $\sum_{n\leq x} r_{A,(1,1)}(n)^2=O(x)$ via a construction of Ruzsa \cite{ruz90}. Whether $A(x)\gg x^{1/2}$ forces $r_{A,(1,1)}(n)$ to be unbounded remains open.
 
%%%%%%%%%%%%%%%%%%%%%%%%%%%%%%%%%%%%%%%%%%%
\section{Matched-even case}
 Let $A\subseteq \N$, and let $\ul{b}=(b_1,\dots,b_h)\in \Z^h$ be fixed. For $N\geq 1$ and $x\geq 1$, define
 \[ S_{A,\ul{b}}(x;N):=\#\bigg\{(u_1,\dots,u_h)\in \{a_{N+1},\dots,a_{2N}\}^h ~\bigg|~ \bigg|\sum_{i=1}^h b_i u_i\bigg|\leq x\bigg\}. \]
 Let $S^*_{A,\ul{b}}(x;N)$ denote the same quantity, but with the additional restriction that $u_i\neq u_j$ for $i\neq j$. Since the sets $\{a_{2^m+1},\dots,a_{2^{m+1}}\}$ are pairwise disjoint, we have
 \begin{equation}
  \sum_{|n|\leq x} r_{A,\ul{b}}(n)\geq \sum_{m\geq 0} S^*_{A,\ul{b}}(x;2^m). \label{lbSab}
 \end{equation}
 Moreover, tuples with at least one repeated coordinate contribute at most $O(N^{h-1})$, and hence
 \begin{equation}
  S_{A,\ul{b}}(x;N) = S^*_{A,\ul{b}}(x;N) + O(N^{h-1}). \label{ssstr}
 \end{equation}

 The next lemma gives a lower bound for $S_{A,\ul{b}}(x;N)$ in the matched-even case.

 \begin{lem}\label{matched-shell}
  Let $k\geq 1$ and $\ul{b}=(c_1,-c_1,\dots,c_k,-c_k)$ with $c_i\in \Z\setminus\{0\}$. For every $N\geq 1$ and $x\geq 1$, we have
  \[ S_{A,\ul{b}}(x;N)\gg_{\ul{b}} \frac{xN^{2k}}{\max\{x,\,a_{2N}-a_N\}}. \]
 \end{lem}
  \begin{proof}
  Replacing $x$ by $\lfloor x\rfloor$, we may assume that $x$ is an integer. For $N<n\leq 2N$, let $y_n := a_n-a_N$, and define
  \[ F_N(\alpha) := \sum_{n=N+1}^{2N} e(y_n\alpha), \qquad e(\alpha):=e^{2\pi i\alpha}. \]
  For $t\in \Z$, define
  \[ \psi_x(t) := \max\bigg\{1-\frac{|t|}{x+1},\,0\bigg\}, \]
  so that $\psi_x(t)\leq \mathbbm{1}_{[-x,x]}(t)$, where $\mathbbm{1}_{[-x,x]}(t)=1$ if $|t|\leq x$ and $0$ otherwise. By orthogonality, we obtain
  \begin{align*}
   S_{A,\ul{b}}(x;N)
   &= \sum_{n_1,m_1=N+1}^{2N}\cdots\sum_{n_k,m_k=N+1}^{2N} \mathbbm{1}_{[-x,x]}\bigg(\sum_{j=1}^{k} c_j(a_{n_j}-a_{m_j})\bigg) \\
   &\geq \sum_{n_1,m_1=N+1}^{2N}\cdots\sum_{n_k,m_k=N+1}^{2N} \psi_x\bigg(\sum_{j=1}^{k} c_j(y_{n_j}-y_{m_j})\bigg) \\
   &= \int_{0}^{1} \widehat{\psi_x}(\alpha) \sum_{n_1,m_1=N+1}^{2N}\cdots\sum_{n_k,m_k=N+1}^{2N} e\bigg(-\alpha\sum_{j=1}^{k} c_j(y_{n_j}-y_{m_j})\bigg)\,\mathrm{d}\alpha \\
   &= \int_{0}^{1}\widehat{\psi_x}(\alpha)\prod_{j=1}^{k}|F_N(c_j\alpha)|^2\,\mathrm{d}\alpha,
  \end{align*}
  where $\widehat{\psi_x}(\alpha) := \sum_{t\in\Z}\psi_x(t)\,e(t\alpha)$. Since
  \[ \widehat{\psi_x}(\alpha) = \sum_{t=-x}^{x} \psi_x(t)\,e(t\alpha) = \frac{1}{x+1}\sum_{t=-x}^{x}(x+1-|t|)\,e(t\alpha) = \frac{1}{x+1}\left|\sum_{m=0}^{x}e(m\alpha)\right|^2\geq 0, \]
  the integrand is real and non-negative. Hence, for any $0<\delta\leq 1/2$,
  \begin{equation}
   S_{A,\ul{b}}(x;N)\geq \int_{-\delta}^{\delta}\widehat{\psi_x}(\alpha)\prod_{j=1}^{k}|F_N(c_j\alpha)|^2\,\mathrm{d}\alpha. \label{Sbdelta}
  \end{equation}

  Let $C:=\max_{1\leq j\leq k}|c_j|$, and choose
  \[ \delta := \frac{1}{12C\max\{x,\,a_{2N}-a_N\}}. \]
  If $|\alpha|\leq \delta$, then for every $1\leq j\leq k$ and every $N<n\leq 2N$,
  \[ |2\pi c_j y_n\alpha|\leq 2\pi C\,(a_{2N}-a_N)\,\delta\leq \frac{\pi}{6}, \]
  and therefore
  \[ |F_N(c_j\alpha)|\geq \Re F_N(c_j\alpha)\geq \cos(\pi/6)\,N\gg N. \]
  Also, if $|\alpha|\leq \delta$, then for every $0\leq m\leq x$,
  \[ |2\pi m\alpha|\leq 2\pi x\,\delta\leq \frac{\pi}{6}, \]
  so
  \[ \left|\sum_{m=0}^{x}e(m\alpha)\right|\geq \cos(\pi/6)\,(x+1)\gg x, \]
  and hence $\widehat{\psi_x}(\alpha)\gg x$.

  It follows from \eqref{Sbdelta} that
  \[ S_{A,\ul{b}}(x;N)\gg \delta\,x\,N^{2k}\gg_{\ul{b}} \frac{xN^{2k}}{\max\{x,\,a_{2N}-a_N\}}. \qedhere \]
 \end{proof}

 With Lemma \ref{matched-shell} in hand, the proof of Theorem \ref{MTme} is obtained by summing the lower bound \eqref{lbSab} over a suitable range of $N$.

%%%%%%%%%%%%%%%%%
\subsection{Proof of Theorem \ref{MTme} \ref{me-i}}
 The hypothesis $A(x)/(x/\log x)^{1/2k} \to \infty$ is equivalent to $a_N = o(N^{2k}\log N)$. This means there exists some non-decreasing function $\xi(x)\to \infty$ such that
 \[ a_N\leq \frac{N^{2k}\log N}{\xi(N)}. \]
 Fix $0<\eps<1/2k$, and let $x$ be large. For each $N=2^m$ in the range
 \[ x^{1/2k}\leq N\leq x^{(1-\eps)/(2k-1)}, \]
 we have $x\leq N^{2k}$, and so \eqref{ssstr} together with Lemma \ref{matched-shell} gives
 \begin{align*}
  S^{*}_{A,\ul{b}}(x;N) &\gg \frac{xN^{2k}}{\max\{x,\,a_{2N}-a_N\}}-O(N^{2k-1}) \\
  &\geq \frac{xN^{2k}}{\max\{N^{2k},\,a_{2N}\}}-O(N^{2k-1}) \\
  &\gg x \min\bigg\{1,\frac{\xi(N)}{\log N}\bigg\}-O(N^{2k-1}),
 \end{align*}
 since $\xi$ is non-decreasing. Summing over all $N=2^m$ in this range and using \eqref{lbSab}, we obtain
 \[ \sum_{|n|\leq x} r_{A,\ul{b}}(n)\gg x\sum_{\substack{x^{1/2k}\leq N\leq x^{(1-\eps)/(2k-1)} \\ N=2^m}} \min\bigg\{1,\frac{\xi(N)}{\log N}\bigg\} - O\Bigg(\sum_{\substack{x^{1/2k}\leq N\leq x^{(1-\eps)/(2k-1)} \\ N=2^m}} N^{2k-1}\Bigg). \]
 
 The error term is $O(x^{1-\eps})$, since the summation is over $N=2^m$ and the largest value of $N^{2k-1}$ is $O(x^{1-\eps})$. Also, because $\eps<1/2k$, the interval above contains $\asymp \log x$ such values of $N$, and for each such $N$ we have $\log N\asymp \log x$ and $\xi(N)\geq \xi(x^{1/2k})$. Hence
 \[ \sum_{\substack{x^{1/2k}\leq N\leq x^{(1-\eps)/(2k-1)} \\ N=2^m}} \min\bigg\{1,\frac{\xi(N)}{\log N}\bigg\} \gg \log x \,\min\bigg\{1,\frac{\xi(x^{1/2k})}{\log x}\bigg\}. \]
 Since $\xi(x^{1/2k})\to\infty$, the right-hand side tends to infinity, and therefore
 \[ \frac{1}{x}\sum_{|n|\leq x} r_{A,\ul{b}}(n)\to\infty, \]
 proving \ref{me-i}. \hfill$\square$

%%%%%%%%%%%%%%%%%%%%%%%%%%%
\subsection{Proof of Theorem \ref{MTme} \ref{me-ii}}
 The hypothesis $A(x)\gg x^{1/2k}$ is equivalent to $a_N\ll N^{2k}$, and therefore $a_{2N}-a_N\ll N^{2k}$. Fix $0<\eps<1/2k$, and let $x$ be large. For each $N=2^m$ in the range
 \[ x^{1/2k}\leq N\leq x^{(1-\eps)/(2k-1)}, \]
 we have $x\leq N^{2k}$ and $N^{2k-1}\leq x^{1-\eps}$. Hence \eqref{ssstr} together with Lemma \ref{matched-shell} gives
 \[ S^*_{A,\ul{b}}(x;N)\gg \frac{xN^{2k}}{\max\{x,\,a_{2N}-a_N\}}-O(N^{2k-1})\gg x-O(x^{1-\eps})\gg x. \]
 Because $\eps<1/2k$, the interval above contains $\asymp \log x$ values of $N = 2^m$. Summing over these values and using \eqref{lbSab}, we obtain
 \[ \sum_{|n|\leq x} r_{A,\ul{b}}(n)\gg x\sum_{\substack{x^{1/2k}\leq N\leq x^{(1-\eps)/(2k-1)} \\ N=2^m}} 1\gg x\log x, \]
 proving \ref{me-ii}. \hfill$\square$

%%%%%%%%%%%%%%%%%%%%%%%%%%%%%%%%%%%%%%%%%%%
\section{Gap conditions}
 Since the case $h=2$ is already covered by the matched-even case, we assume throughout this section that $h\geq 3$. Let $A\subseteq \N$, and let
 \[ \Delta_N:=a_{N+1}-a_N,\qquad \widetilde{\Delta}_N:=\max_{N\leq m\leq 2N}\Delta_m. \]
 The main idea of the proof of Theorem \ref{MTbdd} is that, under suitable control of the gaps on a block $[N,2N]$, one can build many representations with value in $[-x,x]$. We first isolate the following idea: if a function decreases in steps of size at most $D$ and changes sign, then it must remain inside $[-x,x]$ for many consecutive values, provided $x$ is large compared to $D$.
 
 \begin{lem}\label{sliding-claim}
  Let $K,D\geq 1$, and let $F:\{0,\ldots,K\}\to \R$ satisfy
  \[ F(0)>0>F(K), \qquad 0<F(\ell)-F(\ell+1)\leq D \qquad (0\leq \ell<K). \]
  Then for $x\geq D$, the set
  \[ I:=\{0\leq \ell\leq K ~|~ |F(\ell)|\leq x\} \]
  is an interval of integers and $\#I \geq \min\{\lfloor x/D\rfloor, K\} + 1$.
 \end{lem}
 \begin{proof}
  Since $F$ is strictly decreasing, $I$ is an interval. Let $v$ be the unique integer with
  \[ F(v)>0\geq F(v+1), \]
  and let $m := \min\{\lfloor x/D \rfloor, K\} \geq 1$. If $0\leq j\leq m-1$ and $v-j\geq 0$, then
  \[ 0<F(v-j)=F(v+1)+\sum_{t=v-j}^{v}(F(t)-F(t+1))\leq (j+1)D\leq mD\leq x. \]
  Likewise, if $0\leq j\leq m-1$ and $v+1+j\leq K$, then
  \[ -x\leq -mD\leq -(j+1)D < F(v)-\sum_{t=v}^{v+j}(F(t)-F(t+1))=F(v+1+j)\leq 0. \]
  Hence every integer in
  \[ [v-m+1,\,v+m]\cap[0,K] \]
  belongs to $I$. If $v\geq m-1$, then $[v-m+1,\,v+1]\subseteq I$, so $\#I\geq m+1$. If $v\leq m-2$, then $[0,m]\subseteq I$, and again $\#I\geq m+1$. This proves the lemma.
 \end{proof}
 
 We now apply Lemma \ref{sliding-claim} to a family of monotone functions attached to each block $[N,2N]$ (see \eqref{funcF}). This will yield the lower bound from which Theorem \ref{MTbdd} follows.

 \begin{lem}\label{dyadic-sliding}
  Let $h\geq 3$, and let $\ul{b} = (b_1,\dots,b_h)\in \mathbb Z^h$ satisfy $\sum_{i=1}^h b_i = 0$ and $b_i\neq 0$ for all $i$. Let $P := \frac{1}{2} \sum_{i=1}^{h} |b_i|$, and for $N\geq 1$, $x\geq 1$ put
  \begin{equation}
   m_N(x):= \min\bigg\{\frac{x}{P\widetilde{\Delta}_N},\, N\bigg\}. \label{defmN}
  \end{equation}
  Then
  \[ \sum_{|n|\leq x} r_{A,\ul{b}}(n)\gg_{h} \sum_{\substack{N\geq 1\\ m_N(x)\geq h^2}} N^{h-2}\,m_N(x). \]
 \end{lem}
 \begin{proof}
  After permuting the coordinates of $\ul{b}$, we may write
  \[ \ul{b}=(p_1,\dots,p_r,-q_1,\dots,-q_s), \]
  where $r,s\geq 1$, $r+s=h$, each $p_i,q_j$ is positive, and $P=p_1+\cdots+p_r=q_1+\cdots+q_s$.

  Fix $N\geq 1$ with $m_N(x)\geq h^2$. Since $m_N(x)\leq N$, this implies $N\geq h^2$, and in particular the interval $[N/4h, N/2h]\cap \Z$ is non-empty. Choose integers
  \[ d_1,\dots,d_{r-1},\ e_1,\dots,e_{s-1}\in [N/4h, N/2h]\cap \Z \]
  (not necessarily pairwise distinct), and write $\ul{d}:=(d_1,\dots,d_{r-1})$ and $\ul{e}:=(e_1,\dots,e_{s-1})$. If $r=1$ there are no $d_i$, and if $s=1$ there are no $e_j$. Define
  \[ t_1:=0,\qquad t_i:=d_1+\cdots+d_{i-1}\quad (2\leq i\leq r), \]
  \[ u_1:=0,\qquad u_j:=e_1+\cdots+e_{j-1}\quad (2\leq j\leq s), \]
  and put
  \[ L:=u_s+1,\qquad K:=L+t_r+1=u_s+t_r+2. \]
  For integers $\ell$ with $0\leq \ell\leq K$, let
  \begin{equation}
   F_{N,\ul{d},\ul{e}}(\ell):=\sum_{i=1}^r p_i a_{N+L+t_i}-\sum_{j=1}^s q_j a_{N+\ell+u_j}, \label{funcF}
  \end{equation}
  and define $I_{N,\ul{d},\ul{e}}(x):=\{0\leq \ell\leq K ~|~ |F_{N,\ul{d},\ul{e}}(\ell)|\leq x\}$.

  We first check that Lemma \ref{sliding-claim} applies to $\ell\mapsto F_{N,\ul{d},\ul{e}}(\ell)$. For $0\leq \ell<K$,
  \begin{equation}
   F_{N,\ul{d},\ul{e}}(\ell)-F_{N,\ul{d},\ul{e}}(\ell+1)=\sum_{j=1}^s q_j(a_{N+\ell+u_j+1}-a_{N+\ell+u_j})>0, \label{diffF}
  \end{equation}
  so this function is strictly decreasing. Also,
  \[ N+L+t_i>N+u_j \qquad (1\leq i\leq r,\ 1\leq j\leq s), \]
  and
  \[ N+K+u_j>N+L+t_i \qquad (1\leq i\leq r,\ 1\leq j\leq s). \]
  Since $p_1+\cdots+p_r=q_1+\cdots+q_s$, it follows that
  \[ F_{N,\ul{d},\ul{e}}(0)>0>F_{N,\ul{d},\ul{e}}(K).\footnotemark \]

  Next\footnotetext{This is the only point at which the zero-sum hypothesis is used. Without the identity $p_1+\cdots+p_r=q_1+\cdots+q_s$, the monotone path $\ell\mapsto F_{N,\ul{d},\ul{e}}(\ell)$ need not cross $0$.}
  we bound the successive differences in \eqref{diffF}. Since each $d_i,e_j$ lies in $[N/4h,N/2h]$, we have $u_s\leq \frac{(s-1)}{2h}N$, $t_r\leq \frac{(r-1)}{2h}N$, and therefore
  \[ K=u_s+t_r+2\leq \frac{(h-2)}{2h}N+2. \]
  Hence every index occurring in \eqref{diffF} is at most
  \begin{align*}
   N+K+u_s+1\leq N+\frac{(h-2)}{2h}N+2+\frac{(s-1)}{2h}N &\leq \frac{(4h-4)}{2h}N + 2 \\
   &< 2N+1,
  \end{align*}
  so all these indices lie in $[N,2N]$. Thus
  \[ 0 < F_{N,\ul{d},\ul{e}}(\ell)-F_{N,\ul{d},\ul{e}}(\ell+1)\leq P\widetilde{\Delta}_N. \]

  We also need a lower bound for $K$. Again from the choice of the $d_i,e_j$,
  \[ K = u_s+t_r+2\geq \frac{(h-2)}{4h}N+2\geq \frac{N}{12}+2\geq \frac{m_N(x)}{12}, \]
  since $h\geq 3$ and $m_N(x)\leq N$. Since $m_N(x)\geq h^2\geq 1$, we also have $x\geq P\widetilde{\Delta}_N$. Therefore Lemma \ref{sliding-claim}, applied with
  \[ F(\ell)=F_{N,\ul{d},\ul{e}}(\ell),\qquad D=P\widetilde{\Delta}_N, \]
  yields
  \begin{equation}
   \#I_{N,\ul{d},\ul{e}}(x) > \min\bigg\{\frac{x}{P\widetilde{\Delta}_N}, K\bigg\}\gg m_N(x). \label{lbImN}
  \end{equation}

  We now remove the values of $\ell$ that produce repeated indices. The positive block is automatically distinct because $t_1<\cdots<t_r$, and the negative block is distinct because $u_1<\cdots<u_s$. Thus a repetition can only occur between one index from the positive block and one from the negative, that is, when $N+L+t_i=N+\ell+u_j$ for some pair $(i,j)$. Equivalently,
  \[ \ell = L+t_i-u_j. \]
  For fixed $(N,\ul{d},\ul{e})$, there are at most $rs\leq h^2/4$ such values of $\ell$. Since $m_N(x)\geq h^2$, \eqref{lbImN} implies
  \[ \#I_{N,\ul{d},\ul{e}}(x)-\frac{h^2}{4}\geq \frac{3}{4}\,\#I_{N,\ul{d},\ul{e}}(x)\gg m_N(x). \]
  Hence, for each fixed $(N,\ul{d},\ul{e})$, there are $\gg m_N(x)$ admissible values of $\ell$ giving pairwise distinct indices.

  Each admissible quadruple $(N,\ul{d},\ul{e},\ell)$ yields an ordered $h$-tuple
  \[ (a_{N+L+t_1},\ldots,a_{N+L+t_r},a_{N+\ell+u_1},\ldots,a_{N+\ell+u_s}) \]
  with pairwise distinct indices and
  \[ \bigg|\sum_{i=1}^r p_i a_{N+L+t_i}-\sum_{j=1}^s q_j a_{N+\ell+u_j}\bigg|\leq x. \]
  Moreover, different quadruples give different ordered $h$-tuples: from the differences between successive indices in the positive and negative blocks one recovers $\ul{d}$ and $\ul{e}$, hence $u_s$ and $L$, then $N$ from the first positive index $N+L$, and finally $\ell$ from the first negative index $N+\ell$.

  Finally, let $\mathcal{P}_N$ denote the set of choices of $(\ul{d},\ul{e})$. Since there are $(r-1)+(s-1)=h-2$ independent parameters, each ranging over an interval of length $\asymp_h N$, we have
  \[ \#\mathcal{P}_N\gg_h N^{h-2}. \]
  Therefore, for each fixed $N$ with $m_N(x)\geq h^2$, the number of distinct ordered representations with value in $[-x,x]$ arising from this construction is $\gg \#\mathcal{P}_N\,m_N(x)\gg_h N^{h-2}m_N(x)$. Summing over all such $N$ gives
  \[ \sum_{|n|\leq x} r_{A,\ul{b}}(n)\gg_h \sum_{\substack{N\geq 1\\ m_N(x)\geq h^2}} N^{h-2}\,m_N(x). \qedhere \]
 \end{proof}

 With Lemma \ref{dyadic-sliding} in hand, the proof of Theorem \ref{MTbdd} becomes a matter of controlling the gap parameter $\widetilde{\Delta}_N$ under the hypotheses of \ref{bdd-i} and \ref{bdd-ii}, and then evaluating the corresponding sum over a suitable range of $N$.

%%%%%%%%%%%%%%%%%%%%%%%%%%%%%%%%%%%%%%%%%%%
\subsection{Proof of Theorem \ref{MTbdd} \ref{bdd-i}}
 From the hypothesis $\Delta_N=o(N^{h-1}\log N)$ we also have $\widetilde{\Delta}_N=\max_{N\leq m\leq 2N}\Delta_m=o(N^{h-1}\log N)$. Hence there exists a non-decreasing function $\xi(x)\to \infty$ such that
 \[ \widetilde{\Delta}_N\leq \frac{N^{h-1}\log N}{\xi(N)}. \]

 Fix $0<\eps<1/h$, and let $x$ be large. For each $N$ in the range
 \[ x^{1/h}\leq N\leq x^{(1-\eps)/(h-1)}, \]
 we have $x\leq N^h$, and therefore $x/N^{h-1}\leq N$. Thus, from the definition \eqref{defmN} of $m_N(x)$ 
 \begin{align*}
  m_N(x) = \min\bigg\{\frac{x}{P\widetilde{\Delta}_N},\,N\bigg\} &\geq \min\bigg\{\frac{x\,\xi(N)}{PN^{h-1}\log N},\,N\bigg\} \\
  &\gg \frac{x}{N^{h-1}}\min\bigg\{1,\frac{\xi(N)}{\log N}\bigg\}.
 \end{align*}
 Also, since $N\leq x^{(1-\eps)/(h-1)}$, we have
 \[ \frac{x}{P\widetilde{\Delta}_N}\geq \frac{x\,\xi(N)}{PN^{h-1}\log N}\geq \frac{x^\eps\,\xi(x^{1/h})}{P\log x}, \]
 so $m_N(x)\geq h^2$ in this range for all sufficiently large $x$.

 Lemma \ref{dyadic-sliding} therefore gives
 \begin{align*}
  \sum_{|n|\leq x} r_{A,\ul{b}}(n) &\gg \sum_{x^{1/h}\leq N\leq x^{(1-\eps)/(h-1)}} N^{h-2}\,m_N(x) \\
  &\gg \sum_{x^{1/h}\leq N\leq x^{(1-\eps)/(h-1)}} \frac{x}{N}\min\bigg\{1,\frac{\xi(N)}{\log N}\bigg\} \\
  &\gg x\min\bigg\{1,\frac{\xi(x^{1/h})}{\log x}\bigg\} \sum_{x^{1/h}\leq N\leq x^{(1-\eps)/(h-1)}} \frac{1}{N}.
 \end{align*}
 Since $\eps < 1/h$, we have $\sum_{x^{1/h}\leq N\leq x^{(1-\eps)/(h-1)}} N^{-1} \gg \log x$, and hence
 \[ \frac{1}{x}\sum_{|n|\leq x} r_{A,\ul{b}}(n) \gg \min\{\log x,\,\xi(x^{1/h})\}. \]
 Since $\xi(x^{1/h})\to\infty$, this proves \ref{bdd-i}. \hfill$\square$

%%%%%%%%%%%%%%%%%%%%%%%%%%%%%%%%%%%%%%%%%%%
\subsection{Proof of Theorem \ref{MTbdd} \ref{bdd-ii}}
 From the hypothesis $\Delta_N\ll N^{h-1}$ we also have $\widetilde{\Delta}_N=\max_{N\leq m\leq 2N}\Delta_m\ll N^{h-1}$. Fix $0<\eps<1/h$, and let $x$ be large. For each $N$ in the range
 \[ x^{1/h}\leq N\leq x^{(1-\eps)/(h-1)}, \]
 we have $x\leq N^h$, and therefore $x/N^{h-1}\leq N$. Hence, from the definition \eqref{defmN} of $m_N(x)$,
 \[ m_N(x) = \min\bigg\{\frac{x}{P\widetilde{\Delta}_N},\,N\bigg\} \gg \min\bigg\{\frac{x}{N^{h-1}},\,N\bigg\} = \frac{x}{N^{h-1}}. \]
 Also, since $N\leq x^{(1-\eps)/(h-1)}$, we have
 \[ \frac{x}{P\widetilde{\Delta}_N}\gg \frac{x}{N^{h-1}}\gg x^{\eps}, \]
 so $m_N(x)\geq h^2$ in this range for all sufficiently large $x$.

 Lemma \ref{dyadic-sliding} therefore gives
 \begin{align*}
  \sum_{|n|\leq x} r_{A,\ul{b}}(n) &\gg \sum_{x^{1/h}\leq N\leq x^{(1-\eps)/(h-1)}} N^{h-2}\,m_N(x) \\
  &\gg x\sum_{x^{1/h}\leq N\leq x^{(1-\eps)/(h-1)}} \frac{1}{N}.
 \end{align*}
 Since $\eps<1/h$, we have $\sum_{x^{1/h}\leq N\leq x^{(1-\eps)/(h-1)}} N^{-1}\gg \log x$, and hence
 \[ \frac{1}{x}\sum_{|n|\leq x} r_{A,\ul{b}}(n)\gg \log x, \]
 which proves \ref{bdd-ii}. \hfill$\square$

%%%%%%%%%%%%%%%%%%%%%%
\addtocontents{toc}{\protect\setcounter{tocdepth}{0}}
\section*{Acknowledgements}
 The author acknowledges the support of the São Paulo Research Foundation (FAPESP), Brazil, Process No.~2025/15961-3.
 
\addtocontents{toc}{\protect\setcounter{tocdepth}{1}}
%%%%%%%%%%%%%%%%%%%%%%

% ----------------------------------------------------------------
\bibliographystyle{amsplain}
\bibliography{$HOME/Academie/Recherche/_latex/bibliotheca}%

\providecommand{\bysame}{\leavevmode\hbox to3em{\hrulefill}\thinspace}
\providecommand{\MR}{\relax\ifhmode\unskip\space\fi MR }
\providecommand{\MRhref}[2]{%
  \href{http://www.ams.org/mathscinet-getitem?mr=#1}{#2}
}
\providecommand{\href}[2]{#2}
\begin{thebibliography}{1}

\bibitem{che93}
S.~Chen, \emph{On {S}idon sequences of even orders}, Acta Arith. \textbf{64}
  (1993), no.~4, 325--330.

\bibitem{erdtur41}
P.~Erd\H{o}s and P.~Tur\'an, \emph{On a problem of {S}idon in additive number
  theory, and on some related problems}, J. Lond. Math. Soc. \textbf{16}
  (1941), 212--215.

\bibitem{halberstam83}
H.~Halberstam and K.~F. Roth, \emph{Sequences}, revised ed., Springer, 1983.

\bibitem{jia94}
X.-D. Jia, \emph{On {$B_{2k}$}-sequences}, J. Number Theory \textbf{48} (1994),
  no.~2, 183--196.

\bibitem{nat12}
M.~B. Nathanson, \emph{Thin bases in additive number theory}, Discrete Math.
  \textbf{312} (2012), 2069--2075.

\bibitem{nat22}
\bysame, \emph{The {B}ose--{C}howla argument for {S}idon sets}, J. Number
  Theory \textbf{238} (2022), 133--146.

\bibitem{ruz90}
I.~Z. Ruzsa, \emph{A just basis}, Monatsh. Math. \textbf{109} (1990), 145--151.

\bibitem{ruz98}
\bysame, \emph{An infinite {S}idon sequence}, J. Number Theory \textbf{68}
  (1998), no.~1, 63--71.

\bibitem{taovu06}
T.~Tao and V.~H. Vu, \emph{Additive combinatorics}, Cambridge Stud. Adv. Math.,
  vol. 105, Cambridge Univ. Press, 2006.

\end{thebibliography}
\end{document}